\documentclass[letter,11pt]{amsart}
\usepackage[colorlinks,citecolor=red,urlcolor=blue,bookmarks=false,hypertexnames=true]{hyperref} 
\usepackage{amsmath, amssymb, amsfonts, amsbsy, amsthm, latexsym, epsfig}
\usepackage{epstopdf }
\usepackage{graphicx}
\usepackage[shortlabels]{enumitem}
\usepackage{color}
\usepackage{lmodern}
\usepackage{microtype}
\usepackage{mathtools}
\mathtoolsset{showonlyrefs}
\definecolor{carlos-color}{rgb}{0.65,0.0,0.35}
\definecolor{ale-color}{rgb}{0.0, 0.5, 0.0}
\definecolor{cadmiumgreen}{rgb}{0.0, 0.42, 0.24}
\definecolor{burntorange}{rgb}{0.8, 0.33, 0.0}

\theoremstyle{plain}
\newtheorem{theorem}{Theorem} [section]
\newtheorem{lemma}[theorem]{Lemma}
\newtheorem{proposition}[theorem]{Proposition}
\newtheorem{corollary}[theorem]{Corollary}

\theoremstyle{definition}

\newtheorem*{teo*}{Theorem}

\theoremstyle{definition}
\newtheorem{definition}[theorem]{Definition}

\newtheorem{remark}[theorem]{Remark}

\newcommand{\esssup}{{\mathrm{ess}\sup}}
\newcommand{\cF}{\mathcal{F}}

\newcommand{\cM}{\mathcal{M}}
\newcommand{\cN}{\mathcal{N}}

\newcommand{\TT}{{\mathbb T}}
\newcommand{\cV}{\mathcal{V}}
\newcommand{\cB}{\mathcal{B}}

\newcommand{\cX}{\mathcal{X}}

\newcommand{\HH}{\mathcal{H}}
\newcommand{\KK}{\mathcal{K}}

\newcommand{\cA}{\mathcal{A}}

\newcommand{\R}{\mathbb{R}}
\newcommand{\Z}{\mathbb{Z}}
\newcommand{\N}{\mathbb{N}}
\newcommand{\CC}{\mathbb{C}}

\newcommand{\fN}{\mathcal{N}}

\newcommand{\la} {\lambda}
\newcommand{\var} {\varepsilon}
\newcommand{\hS} {\widehat{S}}

\newcommand{\supp}{{\rm supp\,}}

\newcommand{\spn}{{\rm span}}

\begin{document}
	%\nocite{*}
	
	\title{Frames by orbits of two operators that commute}
	
	\let\thefootnote\relax\footnote{
	
	
	This research  was supported by grants: UBACyT 20020170100430BA, PICT 2018-3399 (ANPCyT), PICT 2019-03968  (ANPCyT) and CONICET PIP 11220110101018.
	Diana Carbajal was supported by the European Union's programme Horizon Europe, HORIZON-MSCA-2021-PF-01, Grant
	agreement  101064206.}

	\author{A. Aguilera}
	\address{ Departamento de Matem\'atica, Universidad de Buenos Aires,
		Instituto de Matem\'atica "Luis Santal\'o" (IMAS-CONICET-UBA), Buenos Aires, Argentina}
	\email{aaguilera@dm.uba.ar}
	
	\author{C. Cabrelli}
	\address{ Departamento de Matem\'atica, Universidad de Buenos Aires,
	Instituto de Matem\'atica "Luis Santal\'o" (IMAS-CONICET-UBA), Buenos Aires, Argentina}
	\email{carlos.cabrelli@gmail.com}
	
	\author{D. Carbajal}
	\address{ Faculty of Mathematics, University of Vienna, Vienna, Austria}
	\email{diana.agustina.carbajal@univie.ac.at}
	
	\author{V. Paternostro}
	\address{ Departamento de Matem\'atica, Universidad de Buenos Aires,
		Instituto de Matem\'atica "Luis Santal\'o" (IMAS-CONICET-UBA), Buenos Aires, Argentina}
	\email{vpater@dm.uba.ar}

	\begin{abstract}
	
	Frames formed by orbits of  vectors through the iteration of a bounded operator have recently attracted considerable attention, in particular due to its applications to dynamical sampling.
In this article, we consider two commuting bounded operators acting on some separable Hilbert space $\HH$.
We completely characterize operators $T$ and $L$  with $TL=LT$ and  sets $\Phi\subset \HH$ such that the collection $\{T^k L^j \phi: k\in \Z, j\in J,  \phi \in \Phi \}$ forms a frame of $\HH$.
This is done in terms of model subspaces of  the space of square integrable functions defined on the torus and having values in some Hardy space with multiplicity. 
The operators acting on these models are the bilateral shift and the compression of the unilateral shift (acting pointwisely). 
This context includes the case when the Hilbert space $\HH$ is a subspace of $L^2(\R)$, invariant under translations along the integers, where the operator $T$ is the
translation by one and $L$ is a shift-preserving operator.

\end{abstract}

	\vspace{1cm}
	\maketitle
	\section{Introduction}

	Let $\HH$ be a separable Hilbert space, and $\{g_i\}_{i\in I} $ a countable set in  $\HH$. In sampling theory, every function $f$ belonging to $\HH$
	needs to be recovered from its samples $\{ \langle f,g_i \rangle \}_{i \in I }$. If we impose a stability condition on the reconstruction, this problem is equivalent to the fact that
	the set $\{g_i\}_{i \in I}$ is a frame of $\HH$. 
			
	In dynamical sampling, it is assumed that the signal evolves in time under an evolution operator $L$ defined on $\HH$ and that we are able 
	to sample the functions $L^t f$ for $t = 0,1,2,...,$ where $L^t f$  represents the \emph{evolved signal} at time $t$.
	Thus, if the samples $\{ \langle f,g_i \rangle \}_{i \in I }$ are insufficient to reconstruct every  function $f$ in the space,
	we hope  to compensate the sparse data by sampling the evolved signal as well,
	that is, considering space-time samples.
	
	In other words, the problem consists in trying
	to recover any function $f$ from the samples $\{ \langle L^t f,g_i \rangle \}_{i \in I,  t=0,1,2,...}$ in a stable way,
	which is equivalent to the fact that $\{(L^*)^t g_i\}_{i,t}$ is a frame of $\HH$. 
	This motivates the study of the structure of orbits of operators in Hilbert spaces, i.e. to give conditions on a Hilbert space $\HH$, a bounded operator $L$ defined on $\HH$ and a set of functions $\Phi \subset \HH$ in order that the system $\{L^j\varphi : j \in J, \varphi \in \Phi\}$ forms a frame of $\HH,$ for some  index set $J.$ 	
	
	The dynamical sampling problem was first presented in \cite{ADK13, ADK15}, taking inspiration from the research of Lu, Vetterli, and their colleagues in \cite{HRV10, LV09, RCLV09}. In \cite{LV09}, the authors analyzed a linear diffusion (heat) equation where the initial state evolves over time. They proposed a trade-off between spatial and temporal sampling, with the goal of reconstructing the initial state using only a few spatial sensors (samples) and compensating with temporal samples. This can be considered as an inverse problem in differential equations and has potential applications in various fields, such as weather forecasting, ecology, and medical imaging between others.

	Recently, there has been a lot of interest in this problem. See \cite{ACCMP17,ACMT14,ADK13,ADK15,AKT18,AP17,CMPP17,CMPP18,CHP,Phi17} for foundations and theoretical developments, and  \cite{AHKL17,AK16,RH21,CT21,HNT21,MMO21,Tan17}
	for applications.	

	In \cite{ACCP21},  dynamical sampling was considered for shift-invariant spaces:
	Let $V \subset L^2(\R)$ be a closed subspace, $T$ the translation operator by  one on $V$, i.e. $T f(x) =f(x-1)$,
	and $L:V\rightarrow V$ a bounded operator. We assume that $TV=V$ (that is, $V$ is {\em shift invariant}) and that $TL=LT$ (that is, $L$ is {\em shift preserving}). The main question  here is  to characterize  $L$ and  $\Phi\subset V$ such that   $\{L^j\varphi: j\in J, \varphi \in \Phi\}$ is a set of  frame generators by integer translations of $V$, which means that 
	 $\{T^k(L^j \varphi): k\in \Z, j\in J,  \varphi \in \Phi \}$ forms a frame of $V.$
	
	For the case where $L$ is normal and $V$ is finitely generated, necessary and sufficient conditions were given in \cite{ACCP21}. The conditions were obtained there defining  a special 	diagonalization for $L$, carefully developed in \cite{ACCP-ADV}, and then applying a result for the finite dimensional dynamical sampling problem.
	The approach works for $L^2(\R^d)$ and even when replacing $\R^d$ with an LCA group.

	In this paper we study a generalization of this case. We consider the set of orbits $\mathcal{F}=\{T^k L^j w_i: k\in \Z, j\in J, i\in I\}$  where $\HH$ is an abstract  separable Hilbert space, $T$ and $L$ are commuting bounded operators acting on $\HH$, and $\{w_i\}_{i\in I}$ is an at most countable set of vectors in $\HH.$
	We do not assume that $T$ is unitary  or that  $L$ is normal. We do assume, however, that $T$ is invertible. Then, we ask whether the set $\mathcal{F}$ forms a frame of $\HH$. 
	
	With this purpose, we define an equivalence relation between tuples of the form $(\HH,T,L,\{w_i\}_{i\in I})$. We say that two tuples are {\it similar} if the operators from one tuple are similar to the operators in the other via a bounded isomorphism between the Hilbert spaces involved. The isomorphism also maps the set of vectors of one into the set of vector of the other. See Section \ref{sec:basic-frame-tuples} for a precise definition.

	Our main result asserts that every tuple which generates a frame is similar to a  tuple of a particular class, the class of {\it basic} tuples.
	See Definition \ref{b-t-unilateral-def}.
	These basic tuples are simpler to study because of their structure and they provide a rich platform to prove properties of these systems. In this way, properties preserved by similarity
	can be obtained on a basic tuple and then the results can be  translated to the abstract setting of the general Hilbert space.
	
	The Hilbert space component of a basic tuple is a closed subspace of $L^2(\TT,H^2_{\KK})$ where $\TT$ denote the unit circle, $\KK$ a separable  Hilbert space, and $H^2_{\KK}$  the Hardy space of $\KK$-valued functions defined on the unit circle. See Section \ref{sec:famework}.
	
	The main idea comes from the observation that the kernel of the synthesis operator of the system $\mathcal{F}$
	can be seen as a closed  subspace $\mathcal{M}$ of  $L^2(\TT,H^2_{\KK})$ that turns out  to be \textit{reducing for $U$} (the bilateral shift) and also \textit{invariant for $\widehat{S}$}, an operator
	that acts pointwisely as the unilateral shift in $H^2_{\KK}$. 
	We then provide a functional representation for all the systems $\mathcal{F}$ that form a frame, in terms of the model space $\mathcal{N}_{\mathcal{M}}=L^2(\TT,H^2_{\KK}) \ominus \mathcal{M}$, with the bilateral shift
	and the compression of $\widehat{S}$ on $\mathcal{N}_{\mathcal{M}}$ as the  commuting operators.

We want to remark that our approach was used previously in \cite{CHP} for the case of  just one operator $L$  acting on $\HH$ on one function, and in \cite{CMS22} for several functions.

Next we consider the following question: assume that $\HH$, $T$ and $L$ are given and satisfy the hypothesis of boundedness and commutativity. Which are {\it all vectors}  $w \in \HH$ such that the iterations form a frame?
We were able to answer this question completely and the characterization is similar to the one obtained before in \cite{CHP} for the iteration of only one operator. However, in our case,  the proof is much more involved  and strongly uses the Helson's characterization of reducing subspaces for the bilateral shift. We finally show that this characterization does not work when we consider iterations of multiple generators.

All the results are proved for the two cases: first when we consider forward iterations of the operator $L$, and second, when we allow negative powers of $L$, in which case, $L$ is invertible. The difference between  these cases is that the corresponding basic tuples live in different vector-valued $L^2$-spaces.

The paper is organized as follows. We first introduce in Section \ref{sec:preliminaries} the notation and  the setting that we need in the paper.  In Section \ref{model_space} we study the properties of basic tuples and prove the characterization of tuples that give frames,  splitting our analysis into two cases: \textit{unilateral tuples}, that is,  when we consider forward iterations of the operator $L$, and \textit{bilateral tuples} when  we   allow integer iterations of $L$. 
Finally, in Section \ref{sec:orbits-single-vector} we consider the problem of characterization of  all vectors that give frames by iterations.

	\section{preliminaries}\label{sec:preliminaries}
		
	Let us begin by introducing some notation and definitions. All the Hilbert spaces considered in this paper are separable.
	We will use the letters $\KK$ and $\HH$ to denote complex  Hilbert spaces. As usual, $\cB(\HH,\KK)$ will denote the set of linear bounded operators from $\HH$ into $\KK$ and $\cB(\HH):=\cB(\HH,\HH)$. We will write $\N_0:= \N\cup \{0\}$ and $\TT$ for the complex unit circle. 
	%If $E$ is a  measurable subset of $\TT$, we will denote by $|E|$ its Lebesgue measure normalized such that $|\TT|=1$.
	For a closed subspace $\cN$ of $\HH$, we will write $\cN^{\perp}$ to denote the orthogonal complement of $\cN$ in $\HH$ and $P_{\mathcal N}$ to denote the orthogonal projection of $\HH$ onto $\mathcal N$.
	
	%\subsection{Invariant and reducing subspaces} 
	Let $\cM$ be a closed subspace of a Hilbert space $\HH$ and $A \in \cB(\HH)$. The subspace $\cM$ is said to be {\it invariant} for $A$ (or $A$-invariant) if $A(\cM)\subseteq \cM$. Furthermore, $\cM$ is called {\it reducing} for $A$ (or $A$-reducing) if $\cM$ and $\cM^\perp$ are $A$-invariant, which is equivalent to say that $\cM$ is invariant for $A$ and $A^*$.
	It is easy to check that $\cM$ is $A$-invariant if and only if $P_{\cM} A P_{\cM} = A P_{\cM}$ and $\cM$ is $A$-reducing if and only if $P_{\cM} A = AP_{\cM}$.
	
	Given $T\in \cB(\HH)$ and a closed subspace $\cM$ of $\HH$, the operator defined by $P_{\cM}T|_{\cM}$ is called the {\it compression} of $T$ to the subspace $\cM$.
	
	Particular attention has been put over the years in the study of invariant and reducing subspaces under the action of sequential bilateral and unilateral shift operators, that is
	\begin{alignat}{2}
		U:\{\dots,a_{-1},(a_0),a_1,\dots\}&\mapsto \{\dots,a_{-2},(a_{-1}),a_0,\dots\}, \quad &&\{a_i\}_{i\in\Z}\subseteq\ell^2(\Z),\\
		S:\{a_0,a_1,a_2,\dots\}&\mapsto \{0,a_{0},a_{1},a_2\dots\}, &&\{a_i\}_{i\in\N}\subseteq\ell^2(\N).
	\end{alignat}
	Via the Fourier transform, these can be naturally represented as operators acting over the functional spaces $L^2(\TT)$ and the Hardy space $H^2$, i.e., the subspace of $L^{2}(\TT)$ consisting of all $f\in L^{2}(\TT)$ whose Fourier coefficients vanish for $n<0$, i.e., 
		\begin{equation}
			H^{2} := \left\{ f\in L^{2}(\TT) : \int_{\TT} f(z)z^{-n}\,dz=0 \text{ for } n<0 \right\}.
		\end{equation}

Similarly, shift operators can be defined with multiplicity higher than $1$. This means that instead of shifting a sequence of complex numbers, the operators shift a sequence with values in a Hilbert space $\KK$, where $\dim(\KK)$ is the multiplicity. The representation of these operators over functional spaces requires the introduction of vector-valued functions.

%\vicky{If $\cN\subseteq \HH$ and $\cM\subseteq \KK$ are closed subspaces and there is some $\Phi\in\cB(\HH, \KK)$ such that $\Phi|_\cN:\cN\to\cM$ is an isomorphism, we then write $\cN\simeq_\Phi\cM$. }		
	
	\subsection{Vector-valued functions and Hardy spaces with multiplicity}\label{sec:famework}
	%\subsection{Vector-valued  $L^2$ and $H^2$ spaces}
	Let  $(\Omega, \mu)$ be a $\sigma$-finite measure space such that $L^{2}(\Omega):=L^{2}(\Omega,\mu)$ is a separable Hilbert space. For our purposes along this paper, we will be interested in the cases when $\Omega = \TT$ or $\TT^{2}$ and $\mu$ is the normalized Lebesgue measure ($\mu(\TT)=1$ or $\mu(\TT^2)=1$, respectively). 
	
	Let $\KK$ be a separable Hilbert space. A vector-valued (or $\KK$-valued) function $f:\Omega\rightarrow \KK$ is said to be measurable if for each $x\in\KK$, the complex-valued function $\omega\mapsto \langle f(\omega),x \rangle_{\KK}$ is measurable on $\Omega$. We denote by $L^{2}(\Omega,\KK)$ the space of all measurable $\KK$-valued functions $f$ such that $\int_{\Omega}\|f(\omega)\|_{\KK}^{2}\,d\omega<\infty$, which is a Hilbert space equipped with the inner product 
	\begin{equation}
		\langle f,g \rangle = \int_{\Omega} \langle f(\omega), g(\omega) \rangle_{\KK}\,d\omega, \qquad f,g\in L^{2}(\Omega,\KK).
	\end{equation}
	
	%\vicky{The  space $L^{2}(\TT,\KK)$ is a particular case of vector-valued Hilbert spaces $L^2((\Omega,\mu) , \KK) $ with  .    }\carlos{( no me parece necesaria, nunca lo usamos otra que para $\TT$ ) }\vicky{lo puse porque en la sec de $L$ inversible usamos $\TT^2$}
	
	Notice that the space $\KK$ is naturally embedded in $L^2(\Omega,\KK)$ when identified with the subset of constant vector-valued functions of $L^{2}(\Omega, \KK)$. That is, given $x\in \KK$ define $\tilde{x}:\Omega\to \KK$ as the function in $L^{2}(\Omega, \KK)$ which is constantly $x$. For the sake of simplicity, we will write $x$ instead of $\tilde{x}$, when it is clear from the context that we mean the constant function.	
	%and $\KK=H^{2}_{\ell^{2}(I)}$ or simply $\ell^{2}(I)$.
	Given an orthonormal basis $\mathcal{B} = \{\var_{i}\}_{i\in I}$ of $\KK$ and a vector-valued function $f\in L^{2}(\TT,\KK)$, we can write
	\begin{equation}
		f(\la) = \sum_{i\in I} \langle f(\la),\var_{i} \rangle_{\KK} \,\var_{i}, \quad \text{a.e. }\la\in\TT.
	\end{equation}
	We call $f_{i} := \langle f(\cdot),\var_{i} \rangle_{\KK}$ the {\it ith coordinate function} of $f$ with respect to $\cB$.
	%From the definition of the norm of $L^{2}(\Omega, \KK)$, 
	It is easily seen that $f_i$  belongs to $L^{2}(\TT)$ for every $i\in I$. 
	
	%Indeed, $f_{i}:\TT\to \CC$ is measurable since so is $f:\TT\rightarrow \KK$. On the other hand, $|\langle f(\la),\var_{i} \rangle_{\KK}| \leq  \|f(\la) \|_{\KK} \|\var_{i} \|_{\KK} = \|f(\la) \|_{\KK} $ a.e. $\la\in\TT$ and thus we have that
	%\begin{equation}
	%	\|f_{i}\|_{L^{2}(\TT)}^{2} = \int_{\TT} | f_{i}(\la)|^{2}\,d\la= \int_{\TT} |\langle f(\la),\var_{i} \rangle_{\KK}|^{2} \,d\la
	%	\leq \int_{\TT}  \|f(\la) \|_{\KK}^{2} \,dz = \|f\|
	%	< \infty.
	%\end{equation}

	The {\it Hardy space with multiplicity $\alpha=\dim(\KK)$} is defined as the closed subspace of $L^{2}(\TT,\KK)$ consisting of all functions $f\in L^{2}(\TT,\KK)$ whose coordinate functions $f_{i}$ belong to the Hardy space $H^{2}$. Note that this is equivalent to say that $H^{2}_{\KK}$ is the set of all functions $f\in L^{2}(\TT,\KK)$ such that $\langle f(\cdot),u \rangle_{\KK} \in H^{2}$ for every $u\in \KK$. We will denote it by $H^{2}_{\KK}:=H^{2}(\TT,\KK)$.
	
	We are ready to formally define the shift operators with multiplicity acting over these spaces.
		\begin{definition}\label{def:U}
		The operator $U:L^{2}(\TT,\KK)\to L^{2}(\TT,\KK)$ defined by 
		\begin{equation}%\label{eq:def_U}
			(U f)(\la) = \la f(\la),\quad  \text{a.e. }\la\in\TT,\,f\in L^{2}(\TT,\KK),
		\end{equation}
		is called the {\it bilateral shift on $L^2(\TT,\KK)$} with multiplicity $\alpha=\dim(\KK)$.
	\end{definition}
	
	\begin{definition}\label{def:S}
		The operator $S: H^{2}_{\KK} \rightarrow H^{2}_{\KK}$ given by the restriction of $U$ to $H^{2}_{\KK}$  is called the \textit{unilateral shift on  $H^{2}_{\KK}$} with multiplicity $\alpha=\dim(\KK)$.
	\end{definition}

	The bilateral shift $U$ is unitary and its adjoint operator is given by $(U^{*} f )(\la) = \overline{\la} f(\la)$, for a.e. $\la\in\TT$ and $f\in L^{2}(\TT,\KK)$.			
	Moreover, since $H^{2}_{\KK}$ is invariant under $U$, then the operator $S$ is an isometry. 
	
	%\subsubsection{Bases of vector-valued functions}\label{sec:vector-valued-bases}
	
	There is a simple way to construct an orthonormal basis of $L^2(\TT,\KK)$ or $H^2_\KK$ through iterations of the shift operators over an orthonormal basis of $\KK$ that we will sketch out now.
	%Given $a\in\KK$   denote by $\tilde{a}$  the vector-valued function  $\tilde{a}:\TT\to\KK$ which is constantly $a$. 
	Let $\cB=\{\var_i\}_{i\in I}$ be an orthonormal basis of $\KK$. It is easy to see that $\{ U^k{\var}_i\,:\,k\in\Z,\, i\in I\}$ is an orthonormal basis of $L^2(\TT,\KK)$. 
	Indeed, the system is orthogonal since for every $k,k'\in\Z$ and $i,i'\in I$
	\begin{equation}
		\langle U^k{\var}_i, U^{k'}{\var}_{i'}\rangle = \int_\TT \langle \la^k\var_i, \la^{k'}\var_{i'} \rangle_\KK \,d\la = \langle \var_i, \var_{i'}\rangle_{\KK}\int_{\TT} \la^{k-k'}\,d\la = \delta_{i,i'} \delta_{k,k'}.
	\end{equation}
	For the completeness, note that if  $f\in L^2(\TT,\KK)$ and $f$ is orthogonal to each element in the basis $\cB$, we have
	\begin{equation}
		0=\langle f, U^k \var_i\rangle=\int_\TT \langle f(\la),\var_i\rangle_\KK\la^{-k}\,d\la,\quad\forall k\in\Z,\,\forall i\in I.
	\end{equation}
	Thus, the coordinate functions of $f$ respect to the basis $\cB$, $f_i(\la)= \langle f(\la),\var_i\rangle_\KK$, are zero for a.e. $\la\in\TT$, implying that $f=0$. 
	
	Analogously, we can see that the system $\{S^j{\var}_i\,:\,j\in\N_0,\, i\in I\}$ is an orthonormal basis of~$H^2_\KK$.

\section{Systems of iterations}\label{model_space} 

%In this section we provide a characterization of all systems formed by a Hilbert space $\HH$, two commuting  operators $T$ and $L$ and  a family of vectors $\{w_i\}$ such that
%the iterations $ \left\{ T^{k}L^{j} w_{i}: k\in \Z, j\in J,i\in I\right\}$ form a frame of $\HH.$  We give a representation of these systems in  a kind of {\it universal or basic} systems  given by a class of model spaces, with two shift operators acting on them and a collection of elements in these  model spaces. 
%We show that  every system is similar to a basic system in a sense that will be described below.
%These universal systems have much more structure, in fact the Hilbert space is an $L^2$ space, and properties invariant by similarity can be
%proved here and then translated to the abstract systems.
		
% In this section we provide a characterization of all systems formed by a Hilbert space $\HH$, two commuting  operators $T$ and $L$ and  a family of vectors $\{w_i\}_{i\in I}$ such that
%the iterations $ \left\{ T^{k}L^{j} w_{i}: k\in \Z, j\in J,i\in I\right\}$ form a frame of $\HH.$   We will see that these systems are associated to a particular class of systems of iterations that we will call {\it basic}, which  have much more structure and  consist of two shift operators acting on a vector-valued $L^2$-subspace. 
%
%	To start, let us consider the family of all tuples $(\HH,T,L,\{w_i\}_{i\in I})$ where
%$\HH$ is a Hilbert space, $T,L\in \cB(\HH)$ with $T$ invertible, $TL=LT$, $I$ is an at most countable index set, and
% $\{w_i\}_{i\in I}\subset \HH$ is a collection of vectors.	
 
 	Let $\HH$ be a separable Hilbert space, $T,L\in \cB(\HH)$ such that $T$ is invertible and $TL=LT$, and $\{w_i\}_{i\in I}\subset \HH$ an at most countable collection of vectors.	
 	In this section, we provide a characterization of all systems of iterations $\left\{ T^{k}L^{j} w_{i}: k\in \Z, j\in J,i\in I\right\}$ that form a frame, Parseval frame or Riesz basis of $\HH$. We will see that these systems are associated to a particular class of systems of iterations called {\it basic}, which have much more structure and consist of two shift operators acting on a vector-valued $L^2$-subspace.
 	
 	For this purpose, let us consider the family of all tuples $(\HH,T,L,\{w_i\}_{i\in I})$ with $\HH,T,L$ and $\{w_i\}_{i\in I}$ as above. 
 	We will say that two tuples  $(\HH_1,T_1, L_1,\{v_i\}_{i\in I})$ and $(\HH_2,T_2,L_2,\{w_i\}_{i\in I})$ are {\it similar} if there exists a bounded isomorphism
 $C: \HH_1\rightarrow \HH_2$ that satisfies 
 $C(v_i)=w_i$ for every $i\in I$, and has the following intertwining properties: $$ T_2C=CT_1\quad\textrm{and}\quad L_2C=CL_1.$$
 It can be easily seen that the similarity is an equivalence relation. 
% When $C$ is also isometric, we will say that $(\HH_1,T_1, L_1,\{v_i\}_{i\in I})$ and $(\HH_2,T_2,L_2,\{w_i\}_{i\in I})$ are {\it unitarily equivalent}.
 Moreover, two similar tuples are said to be {\it unitarily equivalent}
 if the associate isomorphism $C$ is unitary. 
 
Furthermore, we will say that a tuple $(\HH,T,L,\{w_i\}_{i\in I})$  is a {\it frame-tuple} (\textit{Parseval-tuple} or \textit{Riesz-tuple}) if the collection 
 \begin{equation}
		 \left\{ T^{k}L^{j} w_{i}: k\in \Z, j\in J,i\in I\right\}
	\end{equation}
	forms a frame (a Parseval frame or a Riesz basis, respectively) of $\HH.$ In this case, the vectors  $\{w_{i}\}_{i \in I}$ will be called \textit{generators}. The index set $J$ will be $\N_0$ or $\Z$. 
	
	When the iterations of $L$ are taken over $\N_{0}$, we call the tuple $(\HH,T,L,\{w_i\}_{i\in I})$ a \textit{unilateral tuple}. Observe that when the iterations are taken over $\Z$, the operator $L$ must be invertible. In that case we call $(\HH,T,L,\{w_i\}_{i\in I})$ a \textit{bilateral tuple}.
	%In what follows we only consider tuples such that the system \eqref{eq:frame_iterations} is Bessel.
	
	The following two lemmas hold regardless of the type of tuple (unilateral or bilateral). 
	\begin{lemma}\label{similar_frames}
		The similarity relation between tuples preserves the frame property.
	\end{lemma}
	
	\begin{proof}
		Assume that the tuples $(\HH_1,T_1,L_1,\{v_i\}_{i\in I}) \text{ and } (\HH_2,T_2,L_2,\{w_i\}_{i\in I})$ are similar via the isomorphism $C\in\cB(\HH_1,\HH_2).$
		Let $g\in \HH_1$. It is sufficient to observe that
		\begin{align*}
			\sum_{k,j,i} |\langle g, T_2^{k}L_2^{j} w_{i} \rangle|^2
			&= \sum_{k,j,i} |\langle g, (CT_1^{k}C^{-1})(CL_1^{j}C^{-1})C v_{i} \rangle|^2 \\ 
			&= \sum_{k,j,i} |\langle C^{*}g, T_1^{k}L_1^{j} v_{i} \rangle|^2		
		\end{align*}
		where the sum is indexed over $k\in\Z,\,j\in J$, where $J=\N_0$ or $\Z$ depending on the case (unilateral or bilateral), and $i\in I$. Since $C$ is a bounded isomorphism, 
		$(\HH_1,T_1,L_1,\{v_i\}_{i\in I})$ is a frame tuple if and only if so is $(\HH_2,T_2,L_2,\{w_i\}_{i\in I})$. 
	\end{proof}

\begin{lemma}\label{lem:similar-parseval-unitarilyequivalent}
		If two Parseval-tuples are similar, then they are unitarily equivalent.
	\end{lemma}
	\begin{proof}
		Let $(\HH_1,T_1,L_1,\{v_i\}_{i\in I})$ and $(\HH_2,T_2,L_2,\{w_i\}_{i\in I})$ be two similar Parseval-tuples via the isomorphism $C\in\cB(\HH_1,\HH_2)$. Then, by the intertwining properties of $C$, we have that
		$C(T_1^kL_1^jv_i) = T_2^kL_2^jw_i$, for every $k\in\Z$, $j\in J$, where $J=\N_0$ or $\Z$, and $i\in I$. 
		In particular, $C$ sends a Parseval frame into another Parseval frame, thus, it must be unitary. 
		To see this, assume that $C\in\cB(\HH_1,\HH_2)$ is an isomorphism and $C(f_k)=g_k$, where $\{f_k\}_k$  and $\{g_k\}_k$ are Parseval frames of the Hilbert spaces $\HH_1$ and $\HH_2,$ respectively. Set  $\widetilde{C}:=C^{-1}$.
		Then we have:
		$$
		\|f\|^2= \sum_k |\langle f,f_k\rangle|^2 =  \sum_k |\langle f,\widetilde{C}C f_k\rangle|^2 =\sum_k |\langle \widetilde{C}^*f,g_k\rangle|^2 =
		\|\widetilde{C}^*f\|^2.
		$$
		Thus, $\widetilde{C}^*$ is unitary, and therefore so is $C$. This completes the proof.
	\end{proof}

	Our goal is to identify all frame-tuples, Parseval-tuples and Riesz-tuples, up to similarity. Let us begin by presenting a particular class of frame-tuples which will be crucial for our purpose.
	
	\subsection{ Basic tuples}\label{sec:basic-frame-tuples}% in $L^{2}(\TT,H^{2}_{\ell^{2}(I)})$}\label{sec:basic-frame-tuples}
	In this subsection we define two types of tuples, the {\it unilateral basic tuple} and the {\it bilateral basic tuple}.
	
	For the unilateral case, we will consider subspaces of $L^{2}(\TT,\KK)$ with $\KK=H^{2}_{\ell^{2}(I)}$, where $I$ is an at most countable index set.
	%, that is, the space of measurable functions defined in the circle $\TT,$
	%and having values in the space $H^{2}_{\ell^2(I)}$.
	So, for $f\in L^2(\TT,H^2_{\ell^2(I)})$,  we have $f(\lambda) \in H^2_{\ell^2(I)}$ and $f(\lambda)(z) \in \ell^2(I)$ for a.e. $\lambda,z \in \TT$. 
	Let us define a new operator acting on $L^{2}(\TT, H^{2}_{\ell^{2}(I)})$ which will play the role of the pointwise unilateral shift operator. 
	\begin{definition} Define $\widehat{S}:L^{2}(\TT, H^{2}_{\ell^{2}(I)})\to L^{2}(\TT, H^{2}_{\ell^{2}(I)})$ as
		\begin{equation}%\label{def_hatS}
			\widehat{S}f(\la)=S(f(\la)),\quad \text{a.e. } \la\in \TT,\, f\in L^2(\TT,H^{2}_{\ell^{2}(I)}).
		\end{equation}
		More precisely, 
		for $f\in L^2(\TT,H^{2}_{\ell^{2}(I)})$ and for a.e. $\lambda,z\in\TT$ 
		$$
		\widehat{S}f(\la)(z)=S(f(\la))(z)=zf(\lambda)(z).
		$$
	\end{definition}

	From the definition of the bilateral shift $U\in\cB(L^2(\TT,H^2_{\ell^{2}(I)}))$ and $\widehat{S}$, it is easy to see that they commute with each other. Moreover, $\widehat{S}$ is an isometry, since so is $S$. 
	Observe that with this operator in hand, we can construct an orthonormal basis of $L^{2}(\TT,H^{2}_{\ell^{2}(I)})$ from an orthonormal basis $\{\var_i\}_{i\in I}$ of $\ell^2(I)$. Indeed, as discussed in the previous section, $\{S^j{\var}_i\,:\,j\in\N_0,\, i\in I\}$ is an orthonormal basis of $H^2_{\ell^{2}(I)}$ and it is clear that the same system can be seen inside of $L^{2}(\TT,H^{2}_{\ell^{2}(I)})$ as $\left\{\widehat{S}^j{\var}_i\,:\,j\in\N_0,\, i\in I\right\}$. Hence, the  system 
	\begin{equation}\label{eq:orth-basis-L2H2}
		\left\{U^k\widehat{S}^j{\var}_i\,:\,k\in\Z,\,j\in\N_0,\, i\in I\right\}
	\end{equation} is an orthonormal basis of $L^{2}(\TT,H^{2}_{\ell^{2}(I)})$. Observe that \eqref{eq:orth-basis-L2H2} is a Fourier basis, $(U^{k} \widehat{S}^{j} \var_{i})(\la)(z)= \la^{k} z^{j} \var_{i}$ for a.e. $\la,z\in \TT$ and $i\in I$.

	\begin{definition} \label{b-t-unilateral-def}
	A tuple $(\cN,U|_{\cN},A,\{\varphi_i\}_{ i\in I})$ is called a {\it unilateral basic tuple} if
	 $\cN$ is a closed subspace  of  $L^{2}(\TT,H^{2}_{\ell^{2}(I)})$ that is $U$-reducing and $\hS^*$-invariant, $A: \cN\rightarrow \cN$
	 is the compression of $\widehat{S}$ to $\cN$, i.e., $A:= P_{\cN}\widehat{S}|_{\cN},$  and $\varphi_i:=P_{\cN}\var_i,$
	 where $\{\var_i\}_{i\in I}$ is the canonical orthonormal basis of $\ell^2(I).$
	\end{definition}
	
	Observe that in Definition \ref{b-t-unilateral-def} we are assuming that $U|_{\cN}$ commutes with $A$. Indeed, for every $f\in \cN$ 
	$$U|_\mathcal N A f= U|_\mathcal NP_{\cN} \hS f= P_{\cN}U|_\mathcal N\hS f = P_{\cN}\hS U|_\mathcal N f= AU|_\mathcal N f.$$
	Here we used that $U$ and $P_{\cN}$ commute because $\cN$ is $U$-reducing, and that $U$ and $\hS$ also commute.
	
	We now turn to the bilateral case. These tuples will be constructed from subspaces belonging to $L^2(\TT^2,\ell^2(I))$ that are reducing for the bilateral shift in each variable, defined as follows. For $i=1,2$,  let $ U_{i}:L^{2}(\TT^{2},\ell^{2}(I))\to L^{2}(\TT^{2},\ell^{2}(I))$ be the operator given by 
 \begin{equation}\label{def:Ui}
	(U_{i}f)(z_{1},z_{2}) = z_{i} f(z_{1},z_{2}),\quad\textrm{a.e. }(z_1,z_2)\in\TT^2.
 \end{equation}
Observe that $U_1$ and $U_2$ commute and $ \{ U_{1}^{k} U_{2}^{j} \var_{i} : k,j\in \Z, i\in I\}$ is an orthonormal basis of $L^{2}(\TT^{2},\ell^{2}(I))$.

\begin{definition}\label{b-t-bilateral-def}
	 A tuple $(\cN,U_1|_\cN,U_2|_\cN, \{\varphi_i \}_{i\in I} )$ is called a {\it bilateral basic tuple}  if
 $\cN$ is a closed subspace of $L^2(\TT^2,\ell^2(I))$ that is reducing for $U_1$ and $U_2$ and  $\varphi_i:=P_{\cN}\var_i,$ where
  $\{\var_i\}_{i\in I}$ is the canonical orthonormal basis of $\ell^2(I)$. 
\end{definition}

	Basic tuples are always frame-tuples. In fact, they are Parseval-tuples as we show next.
	
	\begin{proposition}\label{basic-frame-tuple}	
		Every basic tuple is a Parseval-tuple.
	\end{proposition} 
	\begin{proof} 
	
	\textit{Unilateral case.}
		We first note that the restriction of $U$ to $\cN$, $U|_\cN:\cN\to\cN$, is unitary. To prove the proposition we have to see that the system 
\begin{equation}\label{eq:basic-tuple}
\{U|_{\cN}^kA^j\varphi_i: k\in\Z, j \in\N_0, i\in I\}
\end{equation}		
 is a Parseval frame of $\cN$. We observe that \eqref{eq:basic-tuple} is the orthogonal projection onto $\cN$ of the 
		orthonormal basis $\{U^{k} \widehat{S}^{j} \var_{i} : k\in\Z,j\in \N_{0}, i\in I \}$. Indeed,  since $\cN$ is $\widehat S^*$-invariant, we have $\widehat S^*P_\cN=P_\cN \widehat S^*P_\cN$ and then 
		$ P_\cN\widehat S=P_\cN \widehat SP_\cN$. As $\cN$ is also $U$-reducing it  follows that 
		\begin{equation}\label{eq:intertwine-projection}
			P_{\fN} U^{k} \widehat{S}^{j} \var_{i} 
			=   U^{k} P_{\fN} \widehat{S}^{j} \var_{i}
			=   U^{k} P_{\fN} \widehat{S}^{j} P_{\fN} \var_{i}  
			= U|_\mathcal N^{k} A^j\varphi_{i}.
		\end{equation}
		Thus, the proposition for the unilateral case is proved.
	%\end{proof}
	
%It turns out that, as well as in Section \ref{sec:basic-frame-tuples}, these basic tuples are frame-tuples. 
%
%\begin{lemma}\label{bilateral_frame-tuple}
%$(\cN,U_1|_\cN,U_2|_\cN, \{P_{\fN}\var_{i} \}_{i\in I} )$ is a bilateral frame-tuple.
%\end{lemma}

%\begin{proof} 
\textit{Bilateral case.} This is very similar to the unilateral case. We just need to see that $\{U_1|_{\cN}^kU_2^j|_\cN\varphi_i: k\in\Z, j \in\Z, i\in I\}$ is a Parseval frame of $\cN$.  Since $ \{ U_{1}^{k} U_{2}^{j} \var_{i} : k,j\in \Z, i\in I\}$ is an orthonormal basis of $L^{2}(\TT^{2},\ell^{2}(I))$, then
%Since $L^{2}(\TT^{2},\ell^{2}(I))$ is isometrically isomorphic to 
%$L^{2}(\TT,L^2(\TT, \ell^{2}(I)))$, it is easily seen that $ \{ U_{1}^{k} U_{2}^{j} \var_{i} : k,j\in \Z, i\in I\}$ is an orthonormal basis of $L^{2}(\TT^{2},\ell^{2}(I))$ and then, 
$ \{ P_{\fN}U_{1}^{k} U_{2}^{j} \var_{i} : k,j\in \Z, i\in I\}$ is a Parseval frame for $\fN$. But $P_{\fN}U_{1}^{k} U_{2}^{j} \var_{i}=U_{1}|_\cN^{k} U_{2}|_\cN^{j} P_{\fN}\var_{i}$ for every $ k,j\in\Z,\,i\in I$. Therefore, $(\cN,U_1|_\cN,U_2|_\cN, \{\varphi_i \}_{i\in I} )$ is a Parseval-tuple.
 \end{proof}
 
Since basic tuples are Parseval-tuples, then two similar basic tuples must be unitarily equivalent by Lemma \ref{lem:similar-parseval-unitarilyequivalent}. Even more, they must be equal, as show next.
 
 \begin{theorem}\label{thm:similar-basic-equal}
 	If two basic tuples are similar, then they are equal.
 \end{theorem}

\begin{proof}
	We will just prove the unilateral case, as the bilateral case will be analogous. Let  $(\cN_1,U|_{\cN_1}, A_1, \{P_{\cN_1}\var_i\}_{i\in I})$ and $(\cN_2,U|_{\cN_2}, A_2, \{P_{\cN_2}\var_i\}_{i\in I})$ be two similar basic tuples. Then, there exists a unitary isomorphism $C:\cN_1\to\cN_2$ such that $CA_1=A_2C$, $CU|_{\cN_1}=U|_{\cN_2}C$ and $CP_{\cN_1}\var_i=P_{\cN_2}\var_i$ for every $i\in I$. 
	
	 Moreover, by the intertwining properties of $C$ and \eqref{eq:intertwine-projection} we have that 
	\begin{align}
		C(P_{\cN_1}(U^k\widehat{S}^j\var_i)) &= C(U^kA_1^j P_{\cN_1}\var_i)
		=U^kA_2^jC(P_{\cN_1}\var_i) 
		= U^kA_2^jP_{\cN_2}\var_i \\
		&=P_{\cN_2} (U^k\widehat{S}^j\var_i),
	\end{align}
	for every $k\in\Z$, $j\in\N_0$ and $i\in I$. Since $\{U^k\widehat{S}^j\var_i\}_{k,j,i}$ is an orthonormal basis of $L^2(\TT,H^2_{\ell^2(I)})$, then \cite[Proposition 2.6]{HL} implies that $P_{\cN_1}=P_{\cN_2}$, and so $\cN_1=\cN_2$ and the tuples must be equal. 
\end{proof}

	 In what follows we will show that, in fact, \textit{any} frame-tuple $(\HH,T,L,\{w_i\}_{i\in I})$ must be similar to a unique basic tuple.

Since the $U$-reducing and $\widehat{S}^*$-invariant subspaces of $ L^{2}(\TT, H^{2}_{\ell^{2}(I)})$ serve as  models for frames 
of iterations of two commuting operators -one of them acting unilateraly- the study of their structure, that is very rich, is of independent interest. This is done in great detail in \cite{ACCP22}.

 On the other hand, for the bilateral case, the model subspaces are in $L^2(\TT^2,\ell^2(I))$ and they are  reducing for $U_1$ and $U_2$. This implies that they are \textit{multiplication-invariant} subspaces, and their structure has been very well studied, see for instance \cite{BI,BR}.%
 
\subsection{Unilateral frame-tuples}\label{subsec:unilateral_iterations}

Let $\HH$ be a separable Hilbert space. The main purpose of this subsection is to characterize the unilateral frame-tuples in $\HH$. That is, those tuples $(\HH,T,L,\{ w_i\}_{i\in I})$ where
$T,L\in \cB(\HH)$, $T$ is invertible, $TL=LT$, 
$\{ w_i\}_{i\in I}\subset \HH$ is  an at most countable set, and the system 
\begin{equation}\label{eq:frame_iterations}
	\left\{ T^{k}L^{j} w_{i}: k\in \Z, j\in \N_0,i\in I\right\}
\end{equation}
forms a frame for $\HH$. 

Assume the the system \eqref{eq:frame_iterations} is a Bessel sequence. We then consider its associated  \textit{synthesis operator} defined on $L^{2}(\TT,H^{2}_{\ell^{2}(I)})$, that is 
		\begin{equation}\label{synthesis}
			C:L^{2}(\TT,H^{2}_{\ell^{2}(I)}) \rightarrow \HH, \;\;\;\;\;\;
			%Cf =\sum_{k\in\Z} \sum_{j\geq 0} \sum_{i\in I} f_{kj}^{i} T^{k} L^{j} w_{i}.
			Cf =\sum_{k,j,i} f_{kj}^{i} T^{k} L^{j} w_{i},
		\end{equation}
	where $$f^i_{kj}:=\langle f, U^k\widehat{S}^j\var_i\rangle
	= \int_{\TT} \int_{\TT} \langle f(\la)(z),\var_{i} \rangle \, \la^{-k} z^{-j} \, dz \,d\la,$$ for $k\in\Z, j \in\N_0, i\in I$ (the Fourier coefficients in the basis \eqref{eq:orth-basis-L2H2}).
	Note that, since the system \eqref{eq:frame_iterations} is a Bessel sequence, the series converges unconditionally. 	
	
	\begin{remark}
%	Usually, the synthesis operator of a frame is defined in the sequence space $\ell^{2}(\Z\times \N_0\times I).$
%	We have used here that this sequence space is isometrically isomorphic to $L^{2}(\TT,H^{2}_{\ell^{2}(I)})$ via the Fourier transform.
	Usually, the synthesis operator of a system like in \eqref{eq:frame_iterations} would be defined on the sequence space $\ell^{2}(\Z\times \N_0\times I).$ Here, we precomposed the usual synthesis operator with the isometric  isomorphism between $L^2(\mathbb{T},H^{2}_{\ell^{2}(I)})$ and $ \ell^2(\mathbb{Z} \times \mathbb{N}_{0} \times I)$ defined by $f\mapsto\{\langle f, U^k\widehat{S}^j\var_i\rangle\}_{k,j,i}$.
	\end{remark}

Now we are ready to prove the characterization of unilateral frame-tuples. 

    \begin{theorem}\label{thm:forward ite}

 A tuple $(\HH,T, L,\{w_i\}_{i\in I})$  is a unilateral frame-tuple if and only if it is similar to a unilateral basic tuple. Moreover, this basic tuple is unique.
   \end{theorem}

\begin{proof}
	%[Proof of Theorem \ref{theorem1}]
Assuming that $(\HH,T,L,\{w_i\}_{ i\in I})$ is a unilateral frame-tuple we will find a similar unilateral basic tuple $(\cN,U,A,\{\varphi_i\}_{i\in I})$.

Let $C$ be the synthesis operator defined as in \eqref{synthesis}.
If $f\in L^2(\TT,H^2_{\ell^2(I)})$ we have:
%		Now, consider the operators $U$ as in Definition \ref{def:U}  (with $\KK=H^2_{\ell^2(I)}$) and $\widehat{S}$ as in Definition \ref{def:S-sombrero} (with $\KK=\ell^2(I)$). Observing that for $f\in L^2(\TT,H^2_{\ell^2(I)})$, 
		\begin{equation}
			\langle Uf , U^k\widehat S^j\var_i\rangle = f^i_{k-1,j} \quad\text{and}\quad  \langle U^*f , U^kS^j \var_i\rangle = f^i_{k+1,j},
		\end{equation} 
		for every $(k,j,i)\in\Z\times\N_0\times I$.
		Analogously, we also have  $\langle \widehat{S}f , U^k \widehat{S}^j \var_i\rangle = f^i_{k,j-1}$. Thus,   
		it is easy to see that the following intertwining relations hold:
		\begin{equation}\label{eq:intertwining}
			TC=CU,\quad T^{-1}C = CU^* 	\quad \text{and} \quad  LC=C\widehat{S}.
		\end{equation}
		From here we deduce that $\ker(C)\subseteq L^2(\TT,H^2_{\ell(I)})$ is reducing for $U$ and invariant for $\widehat{S}$ or equivalently,
		$\cN:=\ker(C)^{\perp}$ is reducing for $U$ and invariant for $\widehat{S}^*$.
		
		Given that the system \eqref{eq:frame_iterations} is a frame of $\HH$, $C$ is a bounded surjective operator.
		Then,  the restriction of $C$ to $\cN$ 
		\begin{equation}
			C|_{\cN}:\cN \to \HH,
		\end{equation}
		is a bounded isomorphism. 
		On the other hand, let $\varphi_{i} := P_{\cN}\var_{i}$ for every $i\in I$ and observe that $C(P_{\cN^{\perp}} \var_{i'}) =0$ for every $i'\in I$, and that $\var_{i'}$ is orthogonal to $U^{k} \widehat{S}^{j} \var_{i}$ 	for every $(k,j,i)\in\Z\times\N_0\times I$ except for $k=j=0$ and $i=i'$ in which case $\langle \var_{i'},\var_{i'}\rangle=1$. Then, we obtain
		\begin{equation}
			C|_{\cN} (\varphi_{i'})
			= C \var_{i'} = \sum_{k,j,i} \langle \var_{i'}, U^{k} \widehat{S}^{j} \var_{i}   \rangle T^{k} L^{j} w_{i} = w_{i'}, \quad \text{ for every } i'\in I.
		\end{equation}
	
		Since relations in \eqref{eq:intertwining} still hold when restricting $C$ to $\cN$, we conclude that $(\HH,T,L,\{ w_i\}_{ i\in I})$ is similar to $(\cN,U,A,\{\varphi_i\}_{ i\in I})$. 		
		
	The converse is a consequence of Lemma \ref{similar_frames} and Proposition \ref{basic-frame-tuple}.
The unicity follows from Theorem \ref{thm:similar-basic-equal}. That is, if there exists another basic tuple similar to $(\HH,T,L,\{w_i\}_{ i\in I})$, then by transitivity it must be similar to $(\cN,U,A,\{\varphi_i\}_{ i\in I})$, and therefore equal.
\end{proof}

If $\#I=1$ in Theorem \ref{thm:forward ite}, then $\ell^{2}(I)=\CC$ and $H^{2}_{\ell^{2}(I)} = H^{2}$. Thus, we have the following:

\begin{corollary}\label{cor:forward 1generator}
	A tuple $(\HH,T,L,w)$ is a unilateral frame-tuple if and only if there exists a $U$-reducing and $\widehat{S}^{*}$-invariant subspace $\cN\subseteq L^{2}(\TT,H^{2})$ such that $(\HH,T,L,w)$ is similar to $(\cN,U,A, \varphi)$, where $A$ is the compression of $\widehat{S}$ to $\cN$ and $\varphi=P_{\cN}1$.
	\end{corollary}

Another implication of Theorem \ref{thm:forward ite} is the characterization of Parseval-tuples. 

\begin{corollary}
	A tuple $(\HH,T,L,\{w_{i}\}_{i\in I})$ is a unilateral Parseval-tuple if and only if it is  unitarily equivalent to a unilateral basic tuple.
\end{corollary}

\begin{proof}
Let us assume that $(\HH,T,L,\{w_{i}\}_{i\in I})$ is a unilateral Parseval-tuple. Then, from Theorem \ref{thm:forward ite}, it is similar to a 
basic tuple, that by Proposition \ref{basic-frame-tuple} is a Parseval-tuple. Thus,  by Lemma \ref{lem:similar-parseval-unitarilyequivalent}, they must be unitarily equivalent.
The other implication is immediate since Parseval frames are preserved under unitary operators.
\end{proof}

In the next result, we show that {\it any} unilateral Riesz-tuple must be  similar to 
a basic tuple where the Hilbert space is the whole $L^{2}(\TT,H^{2}_{\ell^{2}(I)})$. 

\begin{proposition}
A tuple $(\HH,T,L,\{w_{i}\}_{i\in I})$ is a unilateral Riesz-tuple if and only if it is similar to $(L^{2}(\TT,H^{2}_{\ell^{2}(I)}), U, \widehat{S}, \{\var_{i}\}_{i\in I})$, where $\{\var_{i}\}_{i\in I}$ is the canonical basis of $\ell^2(I)$.
\end{proposition}

\begin{proof}
If   $(\HH,T,L,\{w_{i}\}_{i\in I})$ is a unilateral Riesz-tuple, the kernel of the synthesis operator is the trivial subspace $\{0\}.$
Thus, its orthogonal complement is $L^{2}(\TT,H^{2}_{\ell^{2}(I)})$ and by the proof of Theorem \ref{thm:forward ite} we obtain the claim in the proposition. 
The converse is straightforward.
\end{proof}

\subsection{Bilateral frame-tuples}\label{L_invertible} 
%\carlos{ Es verdad, quiz\'as se podria decir algo como: "The results for the bilateral case are similar and its poofs use the same arguments than the unilateral case with the obvious differences, so, we will just state the theorems without proof. "}

When the operator $L$ is invertible, it is possible to consider integer iterations of it. Given a tuple $(\HH,T,L,\{ w_i\}_{i\in I})$ with $L$ being invertible, we seek for the frame condition on the set 
$$
	 \left\{ T^{k}L^{j} w_{i}: k,j \in \Z,i\in I\right\}.
$$

As well as in the previous subsection, we obtain  that the set of iterations of $T$ and $L$ generated by $\{ w_i\}_{i\in I}$ is a frame for $\HH$ if and only if the tuple 
$(\HH,T,L,\{ w_i\}_{i\in I})$ is similar to a unique bilateral basic tuple. This is proved by following the same arguments  as in the proof of Theorem \ref{thm:forward ite} replacing the synthesis operator by the corresponding one for this case, that is, 
$C:L^{2}(\TT^{2},\ell^{2}(I)) \rightarrow \HH$  \begin{equation*}
		Cf = \sum_{k,j,i} f_{kj}^{i} T^{k} L^{j} w_{i}
	\end{equation*}
with $f_{kj}^{i}=\langle f, U_{1}^{k} U_{2}^{j} \var_{i} \rangle, $ for $ k,j\in \Z, i\in I$. Hence, we just state the following theorem without proof.

	\begin{theorem}\label{thm:bilateral ite}

	 A tuple $(\HH,T, L,\{w_i\}_{i\in I})$  is a bilateral frame-tuple if and only if it is similar to a bilateral basic tuple. Moreover, this basic tuple is unique.
	\end{theorem}

As in the case of unilateral iterations, we have the following particular case of Theorem \ref{thm:bilateral ite} when $\#I=1$.

\begin{corollary}\label{cor:bilateral 1generator}
A tuple $(\HH,T,L,w)$ is a bilateral frame-tuple if and only if there exists a subspace $\cN\subseteq L^{2}(\TT^{2})$ which is reducing for $U_{1}$ and $U_{2}$ such that $(\HH,T,L,w)$ is similar to  $(\cN,U_{1}|_{\fN},U_{2}|_{\fN}, P_{\fN} 1 )$, where $1\in L^{2}(\TT^{2})$ is the function which is constantly one.
\end{corollary}

Also, we have the next results regarding Parseval-tuples and Riesz-tuples. 
 \begin{corollary}
	A tuple $(\HH,T,L,\{w_{i}\}_{i\in I})$ is a bilateral Parseval-tuple if and only if it is unitarily equivalent to a bilateral basic tuple.
\end{corollary}

\begin{proposition}
	A tuple $(\HH,T,L,\{w_{i}\}_{i\in I})$ is a bilateral Riesz-tuple if and only if it is similar to $(L^{2}(\TT^2,\ell^{2}(I)), U_1, U_2, \{\var_{i}\}_{i\in I})$, where $\{\var_i\}_{i\in I}$ is the canonical basis of $\ell^2(I)$.
\end{proposition}
	
\section{Frames of orbits of a single vector}\label{sec:orbits-single-vector}

Let $\HH$ be a separable Hilbert space, $T,L\in \cB(\HH)$ such that  $T$ is invertible and $TL=LT$. Assume that $\HH,T, L$ and the index set $I=\{1,\dots,n\}$ are fixed and consider the set
\begin{equation}\label{eq:V_n}
	\cV_n:= \big\{ \{v_i\}_{i\in I }\subset\HH: (\HH, T, L, \{v_i\}_{i\in I}) \text{ is a unilateral (bilateral) frame-tuple}\big\}.
\end{equation}
Now, define an equivalent relation in $\cV_n$ given by $\{v_i\}_{i\in I} \sim \{w_i\}_{i\in I}$ if and only if $(\HH, T, L, \{v_i\}_{i\in I})$ is similar to $(\HH, T, L, \{w_i\}_{i\in I})$.

In this section we will prove that when $\# I=1$
%, in which case we will simply denote the set as $\cV_I=\cV$, 
there is only one equivalence class (whenever $\cV_1$ is not empty).  We will later show  that this is not true for multiple generators (see Remark \ref{rem:final-remark}). Specifically, for the unilateral case the claim read as follows.
\begin{theorem}\label{characterization_vectors}
	Let $\HH$ be a Hilbert space, $T,L\in\cB(\HH)$ such that $T$ is invertible and $LT=TL$ and consider the set 
	$$\cV:= \left\{ v\in\HH: (\HH, T, L, v) \text{ is a unilateral frame-tuple}\right\}.$$
	Assume that $w\in\cV.$ Then, $v\in\cV$ if and only if $(\HH, T, L, v)$ is similar to $(\HH, T, L, w)$,
	i.e. $$\cV= \left\{ Bw: B \in \cB(\HH),\,B \text{ is invertible and commutes with } T \text{ and }  L \right\}.$$
\end{theorem}
	
%This is not true for multiple generators  as the following example shows. % Indeed, we can see the next example.
%
%\begin{example}\label{ex:more-classes}
%	Suppose that $\cV_n$ is not empty for some $n>1$. 
%	Let $\{v_i\}_{i\in I } \in \cV_n$, that is, $(\HH,T,L,\{v_i\}_{i\in I})$ is a unilateral frame-tuple, and assume without loss of generality that $\{v_1, v_2\}$ is linearly independent. Then, it is easy to check that $\{v_i\}_{i\in I}\cup \{v_1+v_2\}$ and $\{v_i\}_{i\in I}\cup \{v_1-v_2\}$ belong to $\cV_{n+1}$. Suppose that the corresponding tuples $(\HH,T,L,\{v_i\}_{i\in I}\cup \{v_1+v_2\})$ and $(\HH,T,L,\{v_i\}_{i\in I}\cup \{v_1-v_2\})$ are similar via an isomorphism $C:\HH\to\HH$. Then, in particular, it must hold that 
%	$$C(v_1)=v_1,\; C(v_2)=v_2,\;\dots,\,C(v_1+v_2)=v_1-v_2$$
%	which is not possible.
%\end{example}

In order to prove Theorem \ref{characterization_vectors}, we will show that for $\#I=1$ a weaker condition than similarity is sufficient for two unilateral basic-tuples to be equal.

\begin{theorem}\label{similar_basic_tuples}
	For $i=1,2$, let $\cN_{i} \subseteq L^{2}(\TT,H^{2})$ be a reducing subspace for $U$ and invariant for $\widehat{S}^{*}$, and let $A_{i}=P_{\cN_{i}}\widehat{S}$. If there exists an isomorphism $\Psi:\cN_{1} \rightarrow \cN_{2}$ such that
	\begin{equation}\label{eq:intertwines}
	\Psi A_{1} = A_{2} \Psi \quad \text{and} \quad \Psi U|_{\cN_{1}} = U|_{\cN_{2}} \Psi,
	\end{equation}
	then $\cN_{1}=\cN_{2}$ and $A_{1} = A_{2}$. 
	 
\end{theorem}

Before we prove Theorem \ref{similar_basic_tuples} we need to recall some notions on the structure of reducing subspaces of $L^{2}(\TT, \KK)$, where $\KK$ is a separable Hilbert space. For more details on this topic, we refer the reader to \cite{BR} where the $U$-reducing subspaces of $L^{2}(\TT,\KK)$ are the multiplicative-invariant subspaces of $L^{2}(\TT,\KK)$ with respect to the determining set $\{\gamma^{j} \}_{j\in \Z}$, with $\gamma(\la)= \la$ a.e $\la\in \TT$, see \cite[Definition 2.2 and 2.3]{BR}.

Helson proved in \cite{He} that $U$-reducing subspaces of $L^{2}(\TT,\KK)$ can be characterized in terms of range functions. Later, Bownik and Ross gave in \cite[Theorem 2.4]{BR} an extended version of this result. 

A {\it range function} in $\KK$ is a mapping $J:\TT \rightarrow \{\text{closed subspaces of } \KK\}$. It is measurable if for each $x,y\in\KK$, the complex-valued function $\lambda \mapsto \langle P_{J(\lambda)}x,y \rangle$ is measurable.

\begin{theorem}{\cite[Theorem 2.4]{BR}}\label{thm:MI-spaces}\label{range}
	Let $\cM$ be a closed subspace of $L^{2}(\TT,\KK)$. The following are equivalent:
	\begin{enumerate}[label=\roman*),ref=\roman*)]
		\item $\cM$ is $U$-reducing,
		
		\item there exists a measurable range function such that 
		$$\cM = \{f\in L^{2}(\TT,\KK) : f(\lambda)\in J(\lambda) \textrm{ for a.e. } \lambda\in\TT\}.$$
	\end{enumerate}
	
	The correspondence between $U$-reducing subspaces and measurable range functions is one-to-one and onto, assuming that range functions which are equal almost everywhere are identified.

	Moreover, when $\cM$ is $U$-reducing, since $L^{2}(\TT, \KK)$ is separable, there is a countable set $\cA \subset L^{2}(\TT,\KK)$
	such that $\cM = \overline{\spn}\{U^{k}f : f \in \cA, k \in\Z\}$. Then, the
	measurable range function $J$ associated to $\cM$ is given by 
	\begin{equation}\label{range_func}
		J(\lambda) = \overline{ \spn}\{f(\lambda): f\in  \cA \}.
	\end{equation}
	
\end{theorem}

In \cite{He}, Helson also proved the following property in terms of projections: if $\cM$ is $U$-reducing and $J$ is its range function, then
\begin{equation}\label{proj_Helson}
	P_{\cM}f(\la) = P_{J(\la)}(f(\la)),\quad \text{a.e. } \la\in \TT,\, f\in L^2(\TT,\KK).
\end{equation}

For proving Theorem \ref{similar_basic_tuples} we need a characterization of subspaces of $L^{2}(\TT,H^{2})$  that are reducing for $U$ and invariant for $\widehat{S}^{*}$, that we include below. Its  proof can be found in \cite[Corollary 3.10]{ACCP22}. Recall that a function $h\in H^{2}$ is {\it inner} if $|h(z)|=1$ a.e. $z\in \TT$. 

\begin{theorem}\label{thm:char_Mred_L2(T,H2)}
	Let $\cN \subseteq L^{2}(\TT,H^{2})$ be a closed subspace. The following statements are equivalent:
	\begin{enumerate}[label=\roman*),ref=\roman*)]
		\item  $\cN$ is $U$-reducing and $\widehat{S}^{*}$-invariant,
		
		\item There exists $\phi\in L^{2}(\TT,H^{2})$ such that $\phi(\la)$ is an inner function for a.e. $\la\in\sigma(\cN^{\perp})$ and $\cN^{\perp} = \phi L^{2}(\TT,H^{2}) := \left\{ \phi f : f\in L^{2}(\TT,H^{2}) \right\}$, where $\sigma(\cN^{\perp}) = \{\la \in\TT : J_{\cN^{\perp}}(\la) \neq \{0\}\}$.
	\end{enumerate}
\end{theorem}

Morever, it can be proved that the measurable range function associated to a $U$-reducing subspace of the form $\cM = \phi L^{2}(\TT,H^{2})$ for $\phi\in L^{2}(\TT,H^{2})$ is given by $J_{\cM}(\la) = \phi(\la) H^{2}$ for a.e. $\la\in\TT$ (see \cite[Proposition 3.11]{ACCP22}).

\begin{remark}
A note about the terminology. If $H^2_\KK$ is the Hardy space of $\KK$-valued functions, and 
$\mathcal{Q}\subset  H^2_\KK$ is a non-trivial closed subspace, invariant by $S^*$, the adjoint of the unilateral shift $S$,
then $\mathcal{Q}$ is called a {\it model space}. These spaces had proven to be very important in operator theory since the
compression of the shift on a model space serves as a {\it model } for a certain class of contractions. (See \cite{GMR15} for a 
very comprehensive and detailed introduction to the theory of model spaces).

The Hilbert space  $\cN$ of a basic tuple is a subspace of  $L^{2}(\TT,H^{2}_\KK)$ that is $U$-reducing and $\widehat{S}^{*}$-invariant. After Theorem \ref{range}, we know that $\cN$ has a range function. The invariance under $\widehat{S}^{*}$ tells us that  the  values of the range function are subspaces of $H^2_\KK$ invariant under $S^*$, that is, they are model spaces.

Thus, $\cN$ has a range function whose values are model spaces or, in the terminology of direct integrals, it is the direct integral over $\TT$ of model spaces. This connection justifies calling the space $\cN$, in the $L^2$-context, a model space.
\end{remark}

Another ingredient that we need for the proof of Theorem \ref{similar_basic_tuples} is the concept of {\it operator-valued function}. We give now the definition and some properties.

	An operator-valued function in $\KK$ is a function $F:\TT\to \cB(\KK)$. It is said to be measurable if for every $x\in \KK$, $\lambda\mapsto F(\lambda)(x)$ is measurable. The norm of an operator-valued function in $\KK$ is defined as $\|F\|_{\infty} = {\rm ess\,sup}_{\la\in\TT} \|F(\la)\|_{op}$.
	
	If $F:\TT\to \cB(\KK)$ is a measurable operator-valued function in $\KK$ such that $\|F\|_{\infty}<\infty$, we denote by $\widehat{F}: L^{2}(\TT,\KK) \to L^{2}(\TT,\KK)$ the operator defined by
	$$\widehat{F}f(\la) = F(\la)f(\la), \quad \text{ for a.e. } \la\in\TT, \,\,f\in L^{2}(\TT,\KK)$$
and by $\cF$ the class of all measurable functions $F:\TT\to \cB(\KK)$ such that $\|F\|_{\infty}< \infty$, and $\widehat{\cF} = \{\widehat{F}: F\in\mathcal{F}\}$. 
In \cite[Theorem 3.17]{RR}, it is shown that the correspondence $F\mapsto \widehat{F}$ is an adjoint-preserving algebra isomorphism. Moreover, $\widehat{F}$ is normal  (self-adjoint, unitary or a projection), if and only if $F(\la)$ is normal (self-adjoint, unitary or a projection) for a.e. $\la\in\TT$.
Also, in \cite[Corollary 3.19]{RR} the authors showed that the commutant of $U$ acting on $L^{2}(\TT,\KK)$ is $\widehat{\cF}$. 

We will make use of  the following properties of operator-valued functions and reducing subspaces for $U$ acting on $L^{2}(\TT,\KK)$ whose proofs follows from \cite[Theorem 4.1 and Lemma 4.2]{BI}.

\begin{lemma}\label{lem:prop_hatF}	
	Let $\cM, \cN \subseteq L^2(\TT,\KK)$ be  reducing subspaces for $U$ with range functions $J_\cM, J_\cN$ respectively, and let $\widehat F\in \widehat{\cF}$. Then, we have: 
	\begin{enumerate}[label=\roman*),ref=\roman*)]
				
		\item If $\widehat F|_\cM=0$, then, $F(\la)|_{J_\cM(\la)}=0$ for a.e. $\la\in\TT$.
		
		\item\label{lem:restrictionF_isom} If $\widehat F|_\cM:\cM\to  \cN$ is an isomorphism, then so is $F(\la)|_{J_\cM(\la)}: J_\cM(\la)\to J_{\cN}(\la)$, 
		for a.e. $\la\in\TT$.
		
	\end{enumerate}
	
\end{lemma}

Now we are ready to prove Theorem \ref{similar_basic_tuples}.
	
\begin{proof}[Proof of Theorem \ref{similar_basic_tuples}]
	Since $\cN_{i} \subseteq L^{2}(\TT,H^{2})$ for $i=1,2$ is $U$-reducing and $\widehat{S}^{*}$-invariant, then Theorem \ref{thm:char_Mred_L2(T,H2)} says that there exists $\phi_{i}\in L^{2}(\TT,H^{2})$ such that $\phi_{i}(\la)$ is an inner function for a.e. $\la\in\sigma(\cN_{i}^{\perp})$ and $\cN_{i} = ( \phi_{i} L^{2}(\TT,H^{2}) )^{\perp}$. Thus, its range function $J_{i}$ is given by 
	$$J_{i}: \la \mapsto J_{i}(\la)=( \phi_{i}(\la)H^{2} )^{\perp}.$$
	
	To conclude that $\cN_1=\cN_2$, and hence  $A_1=A_2$,  it is sufficient to show that $\phi_{1}=\phi_{2}$.
		
	Observe first that if we extend $\Psi$ as zero on $\cN_{1}^{\perp}$, from the second equation in \eqref{eq:intertwines} we get that this extension commutes with $U$ and therefore, by  \cite[Corollary 3.19]{RR}, there is $\widehat{F} \in \widehat{\cF}$ such that $\widehat{F} = \Psi$ on $\cN_{1}$ and $\widehat{F} = 0$ on $\cN_{1}^{\perp}$. Moreover, since $\widehat{F}|_{\fN_{1}}=\Psi : \fN_{1} \rightarrow \cN_{2}$ is an isomorphism, \ref{lem:restrictionF_isom} of Lemma \ref{lem:prop_hatF} implies that $F(\la)|_{J_{1}(\la)}: J_{1}(\la) \rightarrow J_{2}(\la)$ is an isomorphism for a.e. $\la \in \TT$.
		
	On the other hand, the first equation in \eqref{eq:intertwines} can be rewritten in terms of $\widehat{F}|_{\fN_{1}}$ as
	\begin{equation}\label{eq:restriccion_Fhat}
		\widehat{F}|_{\fN_{1}} A_{1} = A_{2} \widehat{F}|_{\fN_{1}}.		
	\end{equation}
		
	We want to see now that $A_{1}$ and $A_{2}$ belong to $\widehat{\cF}$. To do this we define for a.e. $\lambda\in\TT$ and for $i=1,2$ the function $G_{i}:\TT\rightarrow \cB(H^{2})$ which for every $\la\in \TT$ is defined as the operator $G_{i}(\la) := P_{J_{i}(\la)} S $. Then, $\la\mapsto G_{i}(\la)$ is measurable since so is $\la\mapsto P_{J_{i}(\la)}$ and $\|G_{i}\|_{\infty} = \esssup_{\la\in\TT} \|G_{i}(\la)\|_{op} = \esssup_{\la\in\TT} \| P_{J_{i}(\la)} S \|_{op} \leq 1 <\infty$. This proves that $G_{i}\in \cF$. 
	
	Besides, for $i=1,2$ we have that $\widehat{G}_{i}$ and $A_{i}$ coincide. This is a consequence of \eqref{proj_Helson} as  the following computation shows: 
	\begin{equation}
		\widehat{G}_{i}f(\la) = G_{i}(\la) (f(\la)) = P_{J_{i}(\la)} (Sf(\la)) = P_{J_{i}(\la)} ((\widehat S f)(\la))  = P_{\cN_{i}} (\widehat{S}f)(\la)
		= A_{i}f (\la)
	\end{equation}
	for every $f\in L^{2}(\TT,H^{2})$ and for a.e $\la\in \TT$. 
		
	Hence, for a.e. $\la\in\TT$ it is obtained from \eqref{eq:restriccion_Fhat} that
	\begin{equation}\label{eq_fibers}
	\begin{split}
	F(\la)|_{J_{1}(\la)} G_{1}(\la) = G_{2}(\la) F(\la)|_{J_{1}(\la)}.
	\end{split}
	\end{equation}

	The equation above says that $F(\la)|_{J_{1}(\la)}$ intertwines $G_{1}(\la)$ and $G_{2}(\la)$ which are two compressions of the unilateral shift operator acting in the subspaces $J_{1}(\la) = (\phi_{1}(\la)H^{2} )^{\perp}$ and $J_{2}(\la)=( \phi_{2}(\la)H^{2} )^{\perp}$ respectively.
	
	At this point we use some results about {\it quasi-similar $C_{0}$ contractions} and {\it minimal functions} which are developed in \cite[Chapter III]{NFBK}.

	First, we observe that by \cite[Proposition 4.3, Chapter III]{NFBK}, for $i=1,2$ and for a.e. $\la\in\TT$, the operator $G_{i}(\la)$ is a $C_{0}$ contraction whose minimal function is $\phi_{i}(\la)$. Second, since $G_{1}(\la)$ and $G_{2}(\la)$ are quasi-similar, i.e, there exists a bounded isomorphism that intertwines $G_{1}(\la)$ and $G_{2}(\la)$ for a.e. $\la\in \TT$, then by \cite[Proposition 4.6, Chapter III]{NFBK} we get that the inner functions $\phi_{1}(\la)$ and $\phi_{2}(\la)$ coincide a.e. $\la \in \TT$ and therefore $\phi_{1}=\phi_{2}$. Consequently, from the definitions of $\cN_{i}$ and $A_{i}$ for $i=1,2$, the result follows.
	\end{proof}
	
%\begin{remark}\label{rem:unicity}\noindent
%	\begin{enumerate}	
%	\item[\rm i)]\label{RM1}Theorem \ref{similar_basic_tuples} implies in particular that if two basic unilateral frame-tuples with one generator are similar, then they have to be equal. To see this, if $(\cN_{1},U,A_{1},\varphi_{1})$ is similar to $(\cN_{2},U,A_{2},\varphi_{2})$, there is an isomorphism $C:\cN_{1}\to \cN_{2}$ such that
%	$$C A_{1} = A_{2} C, \quad CU|_{\cN_{1}} = U|_{\cN_{2}} C,\quad \text{ and } \quad C\varphi_{1} = \varphi_{2}.$$
%	
%	By Theorem \ref{similar_basic_tuples}, the first two equations imply that $\cN_{1}=\cN_{2}$ and $A_{1}=A_{2}$. Moreover, if we recall the definition $\varphi_{i}:=P_{\cN_{i}}1$, $i=1,2$, it follows immediately that $\varphi_{1}=\varphi_{2}$.
%	
%	\item[\rm ii)] Another consequence of Theorem \ref{similar_basic_tuples} is  the unicity in Corollary \ref{cor:forward 1generator}. One has that if $(\fN_{1}, U, A_{1}, \varphi_{1})$ and $(\fN_{2}, U, A_{2}, \varphi_{2} )$ are two basic frame-tuples that are similar to $(\HH, T,L, w)$, then by transitivity the basic tuples are similar, and therefore they coincide by Theorem \ref{similar_basic_tuples}.
%
%	\end{enumerate}
%
%	
%\end{remark}

%The following is the main theorem of this section and its proof relies on Theorem \ref{similar_basic_tuples}.	
Finally, we prove the main results of this section.
	
%\begin{theorem}\label{characterization_vectors}
%	Let $\HH$ be a Hilbert space, $T,L\in\cB(\HH)$ such that $T$ is invertible and $LT=TL$ and consider the set 
%	$$\cV:= \left\{ v\in\HH: (\HH, T, L, v) \text{ is a unilateral frame-tuple}\right\}.$$
%	Assume that $w\in\cV.$ Then, $v\in\cV$ if and only if $(\HH, T, L, v)$ is similar to $(\HH, T, L, w)$,
%	i.e. $$\cV= \left\{ Bw: B \in \cB(\HH),\,B \text{ is invertible and commutes with } T \text{ and }  L \right\}.$$
%\end{theorem}

	\begin{proof}[Proof of Theorem \ref{characterization_vectors}]
	Observe that if $v\in \HH$ is such that  $(\HH,T,L,v)$ is similar to $(\HH,T,L,w)$, then by Lemma \ref{similar_frames}, $(\HH,T,L,v)$ is a frame-tuple. Hence $v\in\cV$. 
		
%\diana{Observe first that $(\HH,T,L,w)$ is similar to $(\HH,T,L,v)$ for some $v\in\HH$. Then, by Lemma \ref{similar_frames} we have that $(\HH,T,L,v)$ is a frame-tuple, and hence $v\in\cV$.} \vicky{Observa that if $v\in \HH$ is such that  $(\HH,T,L,v)$ is similar to $(\HH,T,L,w)$, then by Lemma \ref{similar_frames}, $(\HH,T,L,v)$ is a frame-tuple. Hence $v\in\cV$. }

It remains to show that if $v\in\cV$, then it can be defined a linear bounded operator $B:\HH\to\HH$ with $BT=TB$ and $BL=LB$ such that $v=Bw$.

Since $w$, $v\in \cV$, $(\HH, T,L,w)$ and $(\HH,T,L,v)$ are unilateral frame-tuples. Therefore, by Corollary \ref{cor:forward 1generator}, there exist two basic tuples $(\fN_{1}, U, A_{1},\varphi_{1})$ and $(\fN_{2},U,A_{2}, \varphi_{2})$ that are similar to $(\HH, T,L,w)$ and $(\HH,T,L,v)$ respectively. Recall that the isomorphism $C_{i}\in\cB(\cN_i,\HH)$ for $i=1,2$ given by the similarity relation satisfies that
		\begin{equation}\label{intertwining_C0}
			C_{1}A_{1} = LC_{1}, \quad C_{1} U|_{\cN_{1}} = TC_{1} \quad \text{and} \quad C_{1} \varphi_{1}=w
		\end{equation} 
	and
		\begin{equation}\label{intertwining_C}
			C_{2} A_{2} = L C_{2}, \quad C_{2} U|_{\cN_{2}} = T C_{2} \quad \text{and} \quad C_{2} \varphi_{2} = v.
		\end{equation}
				
%Let $\cT:\HH\rightarrow \cT(\HH)\subseteq L^{2}(\TT,H^{2})$ be a surjective linear isometry as in Proposition \ref{isometry_BI}. By Remark \ref{remark_BI} we have that for $i=1,2$, there exists $\widehat{F}_{i} \in \widehat{\cF}$ defined as $\widehat{F}_{i} := \cT C_{i}$ on $\cN_{i}$ and $\widehat{F}_{i} := 0$ on $\cN_{i}^{\perp}$. Moreover, $\widehat{F}_{i}|_{\fN_{i}} : \fN_{i} \rightarrow \cT(\HH)$ is an isomorphism for $i=1,2$.

%Now, observe that $\cT(\HH)$ is a reducing subspace for $U$ since $U \cT=\cT T$ and $U^{*} \cT  = \cT T^{*} $, then it has a measurable range function, say $K$. Using \ref{lem:restrictionF_isom} of Lemma \ref{lem:prop_hatF}, we have that $F_{i}(\la)|_{J_{i}(\la)}: J_{i}(\la) \rightarrow K(\la)$ is an isomorphism for a.e. $\la \in \TT$. 

Now, from equations \eqref{intertwining_C0} and \eqref{intertwining_C} we have that 
\begin{equation}
	C_{1} A_{1} C_{1}^{-1} = C_{2} A_{2} C_{2}^{-1},
	\quad C_{1} U|_{\cN_{1}} C_{1}^{-1} = C_{2} U|_{\cN_{2}} C_{2}^{-1}. 
\end{equation}
% Thus, applying $\cT$ at the left and $\cT^{-1}$ at the right at both sides of the equations above and writing in terms of $\widehat{F}_{1}$ and $\widehat{F}_{2}$ we get that
%\begin{equation*}
%	\cT C_{1} A_{1}C_{1}^{-1} \cT^{-1} = \cT C_{2} A_{2} C_{2}^{-1}\cT^{-1}.
%\end{equation*}
%%	\begin{equation}
%	\widehat{F}_{1}|_{\fN_{1}} A_{1} \widehat{F}_{1}^{-1}|_{\fN_{1}} = \widehat{F}_{2}|_{\fN_{2}} A_{2} \widehat{F}_{2}^{-1}|_{\fN_{2}},\quad
%	\widehat{F}_{1}|_{\fN_{1}} U \widehat{F}_{1}^{-1}|_{\fN_{1}} = \widehat{F}_{2}|_{\fN_{2}} U \widehat{F}_{2}^{-1}|_{\fN_{2}},
%	\end{equation}
which is equivalent to 
$$\Psi A_{1} = A_{2} \Psi,\quad
\Psi U|_{\cN_{1}} =  U|_{\cN_{2}} \Psi$$
where $\Psi = C_{2}^{-1} C_{1}: \cN_{1}\to \cN_{2}$. Since $\Psi$ is an isomorphism, Theorem \ref{similar_basic_tuples} gives that $\cN_{1}=\cN_{2}$, $A_{1}=A_{2}$ and consequently $\varphi_{1} = \varphi_{2}$.
	
	Now, let $B:\HH\rightarrow \HH$ be defined as $B:= C_{2} C_{1}^{-1}$. Observe that $B$ is bounded, invertible and 
$$v=C_{2}\varphi_{2} = C_{2}C_{1}^{-1} C_{1} \varphi_{1} = B w.$$
 Moreover, 
 	$$BL
	= C_{2} C_{1}^{-1} L = C_{2} A_{1} C_{1}^{-1}
	= C_{2} A_{2} C_{1}^{-1}
		= L C_{2} C_{1}^{-1}
		= LB,$$
		and then, $B$ and $L$ commute. 
		Analogously, it can be seen that $BT=TB$. This completes the proof.
\end{proof}		
	
	%%%%%%%%%%%%%%%%%
	
%\subsection{Single generators for  integer iterations of $L$.}
%
%Our  following theorem is the analogous to Theorem \ref{characterization_vectors}. Define the set $$\cV_{\Z}: = \left\{v\in\HH: (\HH,T,L,v) \text{ is a frame-tuple} \right\}.$$
%
%\begin{theorem}\label{characterization_vectors}
%Let $\HH$ be a Hilbert space, $L,T\in\cB(\HH)$ with $L$ invertible and $T$ unitary such that $LT=TL$ and $w \in \cV_{\Z}$. Then
%$$ \cV_{\Z} = \left\{Bw : B\in\cB(\HH) \text{ invertible, } BL=LB \text{ and } BT=TB \right\}.$$
%\end{theorem}

%\begin{proof}

%In a similar way as in the proof of Theorem \ref{characterization_vectors}, it is shown that every vector of the form $Bw$ with $B$ an invertible operator that commutes with $T$ and $L$, generates a frame of iterations by $T$ and $L$. 
%This means that $$\left\{Bw : B\in\cB(\HH) \text{ invertible, } BL=LB \text{ and } BT=TB \right\} \subseteq \cV_{\Z}.$$	
The same characterization given in Theorem \ref{characterization_vectors} can be obtained when we consider bilateral frame-tuples. More precisely, we have: 

\begin{theorem}\label{chracterization_U1U2}
Let $\HH$ be a Hilbert space, $T,L\in\cB(\HH)$ such that $LT=TL$ and consider the set 
$$\cV:= \left\{ v\in\HH: (\HH, T, L, v) \text{ is a bilateral frame-tuple} \right\}.$$
Assume that $w\in\cV.$ Then, $v\in\cV$ if and only if $(\HH, T, L, v)$ is similar to $(\HH, T, L, w)$,
i.e. $$\cV= \left\{ Bw: B \in \cB(\HH),\,B \text{ is invertible and commutes with } T \text{ and }  L \right\}.$$
%\diana{Assume that $\cV$ is not empty, i.e., there exists some $w\in\cV$ such that $ (\HH, T, L, w)$ is a bilateral frame-tuple.  Then, $v\in\cV$ if and only if $(\HH, T, L, w)$ is similar to $(\HH, T, L, v)$.}
%\carlos{Assume that $w\in\cV.$ Then, $v\in\cV$ if and only if $(\HH, T, L, w)$ is similar to $(\HH, T, L, v)$.}
%\vicky{Assume that $w\in\cV.$ Then, $v\in\cV$ if and only if $(\HH, T, L, v)$ is similar to $(\HH, T, L, w)$.}
\end{theorem}

 The proof of Theorem \ref{chracterization_U1U2} is based in Corollary   \ref{cor:bilateral 1generator} and the following 
 proposition that characterize reducing subspaces of $L^2(\TT^2)$, and whose proof can be found in \cite{GM}.

\begin{proposition}\label{M_reducingU1U2}
	A subspace $\mathcal{M}\subset L^{2}(\TT^{2})$ is reducing for $U_{i}$, $i=1,2$ if and only if there exists a Borel set $E \subseteq \TT^{2}$ such that $\cM = \cX_{E} L^{2}(\TT^{2})$.
\end{proposition}

\begin{proof}[Proof of Theorem \ref{chracterization_U1U2}]
 Let $w\in \cV$. Since $(\HH,T,L,w)$  and $(\HH,T,L,v)$ are bilateral frame-tuples, then by Corollary \ref{cor:bilateral 1generator} they are similar to basic frame-tuples $(\cM, U_{1}|_{\cM},U_{2}|_{\cM}, P_{\cM} 1)$ and $(\fN, U_{1}|_{\fN},U_{2}|_{\fN}, P_{\fN} 1)$ respectively. 

Analogously as in the proof of the case of unilateral iterations, the key is to show that the subspaces $\cM$ and $\fN$ are equal.

From the isomorphisms given by the similarity relations, we can define an isomorphism $V:\cM \rightarrow \fN$ satisfying
\begin{equation}\label{V_intertwines}
	V\,U_{1}|_{\cM} =  U_{1}|_{\fN} V \quad \text{and} \quad V\,U_{2}|_{\cM} =  U_{2}|_{\fN} V.
\end{equation}

Since $\cM$ and $\cN$ are reducing for $U_{1}$ and $U_{2}$, by Proposition \ref{M_reducingU1U2} we have that there exist Borel sets $E_{1}, E_{2} \subseteq \TT^{2}$ such that $\cM = \cX_{E_{1}} L^{2}(\TT^{2})$ and $\fN= \cX_{E_{2}} L^{2}(\TT^{2})$. 

In order to show that $\cM=\fN$, we will see that the characteristic functions $\cX_{E_{1}}$ and $\cX_{E_{2}}$ coincide a.e. on $\TT^{2}$.  To do this, take $g=\cX_{E_{1} \setminus E_{2}} \in \cM \subset L^{2}(\TT^{2})$. Since $\{U_1^kU_2^j1\,:\, k,j\in\Z\}$ is an orthonormal basis of $L^2(\TT^2)$ we have the expansion
$$g= \sum_{k\in\Z}\sum_{j\in\Z} \beta_{kj} U_1^{k} U_2^{j} 1,$$ 
whose coefficients satisfy $\{\beta_{kj}\}\in\ell^{2}(\Z^{2})$. By applying the isomorphism $V$ and using the relations in \eqref{V_intertwines} we get that
\begin{align*}
	Vg= VP_{\cM}g
	&= V \left(  \sum_{k\in\Z}\sum_{j\in\Z} \beta_{kj} \, U_{1}^{k} \,U_{2}^{j} P_{\cM}1 \right) 
	%=  \sum_{k\in\Z}\sum_{j\in\Z} \beta_{kj} V  \,U_{1}^{k}|_{\cM}  \,U_{2}^{j}|_{\cM}(P_{\cM}1)
	= \sum_{k\in\Z}\sum_{j\in\Z} \beta_{kj} \, U_{1}^{k}|_{\fN}  \, U_{2}^{j}|_{\fN} V P_{\fN}1.
	%&= P_{\tilde{\fN}} \left(\sum_{k\in\Z}\sum_{j\in\Z} \beta_{kj} \, \cU_{1}^{k}  \,\cU_{2}^{j} e_{00} \right) V P_{\fN} e_{00}  \\
	%&=(P_{\tilde{\fN}}g) (V P_{\fN} e_{00}) = 0.
\end{align*}

Let $h=V P_{\cN} 1$ and observe that $\supp(g) \subseteq E_{1} \setminus E_{2}$ and $\supp(h)\subseteq E_{2}$. Then, 
%by evaluating in $(z_{1},z_{2}) \in\TT^{2}$ 
it is obtained
$$Vg = gh = 0.$$
%$$Vg(z_{1},z_{2}) = \sum_{k\in\Z}\sum_{j\in\Z} \beta_{kj} z_{1}^{k}  z_{2}^{j} \,h(z_{1},z_{2})
%= g(z_{1},z_{2}) \,h(z_{1},z_{2}) = 0$$
%for a.e. $(z_{1},z_{2})\in\TT^{2}$. 
Since $V$ is an isomorphism, we conclude that $g=0$. A similar argument can be used to show that $\cX_{E_{2} \setminus E_{1}} = 0$. Thus, $| E_{2} \setminus E_{1} |=0= |E_{1} \setminus E_{2} |$, and therefore $\cX_{E_{1}}$ coincide with $\cX_{E_{2}}$ a.e. $(z_{1},z_{2})\in \TT^{2}$. As a consequence, $\cM=\fN$.

The proof concludes in the same manner of Theorem \ref{characterization_vectors}.

\end{proof}

\begin{remark}\label{rem:final-remark}\

\begin{enumerate}[(a)]

\item\label{ex:more-classes} It is not true that in the case of more than one  generator there exists only one equivalent class of vectors $\{v_{i}\} \in \cV_{n}$ as we show next. Suppose that $\cV_n$ is not empty for some $n>1$. 
	Let $\{v_i\}_{i\in I } \in \cV_n$, that is, $(\HH,T,L,\{v_i\}_{i\in I})$ is a unilateral frame-tuple, and assume that $\{v_1, v_2\}$ is linearly independent. Then, it is easy to check that $\{v_i\}_{i\in I}\cup \{v_1+v_2\}$ and $\{v_i\}_{i\in I}\cup \{v_1-v_2\}$ belong to $\cV_{n+1}$. Suppose that the corresponding tuples $(\HH,T,L,\{v_i\}_{i\in I}\cup \{v_1+v_2\})$ and $(\HH,T,L,\{v_i\}_{i\in I}\cup \{v_1-v_2\})$ are similar via an isomorphism $C:\HH\to\HH$. Then, in particular, it must hold that 
	$$C(v_1)=v_1,\; C(v_2)=v_2,\;\dots,\,C(v_1+v_2)=v_1-v_2$$
	which is not possible.

\item As we already mentioned, if we fix the Hilbert space $\HH$ and the operator $T$ and $L$ in $\mathcal B(\HH)$,  Theorems  \ref{characterization_vectors} and  \ref{chracterization_U1U2} 
say that all the frame-tuples with a single generator are similar (when there is at least one). Before, we saw that this is no longer true for more generators. At this point one could ask if there is some relation between $\#I$ and the number of equivalent classes of the relation defined at the beginning of this section. What we know is that if $\#I=1$ and $\mathcal V_1\neq\emptyset$, there is only one equivalence class. But a slight modification in part \ref{ex:more-classes} shows that for $\#I=n>1$ there are infinitely many equivalence classes. Indeed, suppose that $n>1$ and that  $\cV_n\neq\emptyset$. Consider   $\{v_i\}_{i\in I } \in \cV_n$  and assume that $\{v_1, v_2\}$ is linearly independent. Then, for $a,b\in\mathbb C\setminus\{0\}$, we have that $\{v_i\}_{i\in I}\cup \{av_1+bv_2\}\in\cV_{n+1}$. As in \ref{ex:more-classes}, it can be seen that each generator set $\{v_i\}_{i\in I}\cup \{av_1+bv_2\}$ belongs to a different equivalence class. Then, there are infinitely many when we have  3 generators or more. The same idea works for 2 generators, since if $v\in\cV_1$, $\{v,av\}\in\cV_2$ for every $a\in\mathbb C\setminus\{0\}$ and they all belong to different equivalence classes. 

\end{enumerate}
\end{remark}


\begin{thebibliography}{1}

\bibitem{ACCP-ADV}
	A. Aguilera, C. Cabrelli, D. Carbajal, and V. Paternostro,
	Diagonalization of shift-preserving operators.
	Adv. Math. {\bf 389} (2021), Paper No. 107892, 32 pp.

\bibitem{ACCP22} A. Aguilera, C. Cabrelli, D. Carbajal, and V. Paternostro, Reducing and invariant subspaces under two commuting shift operators, preprint (2022), arXiv:2202.10299.
	
\bibitem{ACCP21} A. Aguilera, C. Cabrelli, D. Carbajal, and V. Paternostro,  
Dynamical sampling for shift-preserving operators, Appl. Comput. Harmon. Anal., {\bf 51} (2021), 258--274.


\bibitem{ACCMP17}
A. Aldroubi, C. Cabrelli, A. F. Cakmak, U. Molter, and A. Petrosyan, Iterative actions of normal operators, J. Funct. Anal., {\bf 272}, 3 (2017), 1121--1146.

\bibitem{ACMT14}
A. Aldroubi, C. Cabrelli, U. Molter, and S. Tang, Dynamical Sampling,
Appl. Comput. Harmon. Anal., {\bf 42}, 3 (2017), 378--401.

\bibitem{ADK13}
A. Aldroubi, J. Davis, and I. Krishtal,
Dynamical sampling: time-space trade-off, Appl. Comput. Harmon. Anal., {\bf 34}, 3 (2013), 495--503.

\bibitem{ADK15}
A. Aldroubi, J. Davis, and I. Krishtal, Exact reconstruction of signals in evolutionary systems via spatiotemporal trade-off, J. Fourier Anal. Appl., {\bf 21}, 1 (2015), 11--31.

\bibitem{AHKL17} A. Aldroubi, L. Huang, I. Kristhal, and R. Lederman, Dynamical sampling with random noise. In 2017 International conference on Sampling Theory and Applications (SampTa) (2017), 409--412.

\bibitem{AK16}
A. Aldroubi and I. Krishtal, Krylov subspace methods in dynamical sampling, STSIP, Shannon Centennial Volume, {\bf 15} (2016), 9--20.

\bibitem{AKT18}
A. Aldroubi, I. Krishtal, and S. Tang., Phaseless reconstruction from space--time samples, Appl. Comput. Harmon. Anal., {\bf 48}, 1 (2020), 395--414.

\bibitem{AP17}
A. Aldroubi and A. Petrosyan, Dynamical sampling and systems from iterative actions of operators, Frames and Other Bases in Abstract and Function Spaces, Springer, 2017.


\bibitem{RH21} R. Beinert and M. Hasannasab, Phase Retrieval via polarization in dynamical sampling, In International Conference on Scale Space and Variational Methods in Computer Vision: 8th International Conference, SSVM 2021, Virtual Event (2021), 516-527.

	
%\bibitem{AF} Aubin, J.-P. and Frankowska H., Set-valued analysis, Birkh\"auser Boston-Basel-Berlin (1990).
		
\bibitem{BI} M. Bownik and J. Iverson, Multiplication-invariant operators and the classification of LCA group frames, J. Funct. Anal., {\bf 280}, 2 (2021), Paper No. 108780.
		
\bibitem{BR} M. Bownik and K. A. Ross, The Structure of Translation-Invariant Spaces on Locally Compact Abelian Groups, J. Fourier Anal. Appl., {\bf 21} (2015), 849--884.
	
		
\bibitem{CMPP17} C. Cabrelli, U. Molter, V. Paternostro, and F. Philipp, Finite sensor dynamical sampling, 2017 International Conference on Sampling Theory and Applications (SampTA), (2017), 50--54. 

\bibitem{CMPP18} C. Cabrelli, U. Molter, V. Paternostro, and F. Philipp, Dynamical sampling on finite index sets, J. Anal. Math., {\bf 140},2, (2020), 637--667.

\bibitem{CMS22} C. Cabrelli, U. Molter, and D. Suarez, Frames of iterations and vector valued model spaces, preprint (2022), arXiv:2203.01301.

\bibitem{CT21} J. Cheng, and S. Tang, . Estimate the spectrum of affine dynamical systems from partial observations of a single trajectory data. Inverse Problems, {\bf 38},1 (2021), 015004.

\bibitem{CHP} O. Christensen, M. Hasannasab, and F. Philipp, Frame properties of operator orbits, Math. Nachrichten, {\bf 293}, 1 (2019), 52-66.
		
\bibitem{GMR15} S. R. Garcia, J. Mashreghi, and W. T. Ross, Introduction to Model Spaces and their Operators, Cambridge U.P. (2015).		

\bibitem{GM} P. Gathage and V. Mandrekar, On Beurling type invariant subspaces of $L^{2}(\TT^{2})$ and their equivalence, J. Oper. Theory, {\bf 20}, 1 (1988), 83--89.

\bibitem{Ha} P. Halmos, Shifts on Hilbert spaces, Journal f{\"u}r die reine und angewandte Mathematik (Crelles Journal), {\bf 1961}, (1961), 102--112.

\bibitem{HL} D. Han and  D. R. Larson, Frames, bases and group representations, Mem. Amer. Math. Soc., {\bf 147}, 697 (2000).
		
\bibitem{He} H. Helson, Lectures on Invariant Subspaces, Academic Press, London, 1964.

\bibitem{HRV10} A. Hormati, O. Roy, Y. Lu, and M. Vetterli, Distributed sampling of signals linked by sparse filtering: Theory and applications, IEEE Trans. Signal Process. {\bf 58}
(2010) 1095--1109.

\bibitem{HNT21} L. Huang, D. Needell, and Tang, S. (2021), Robust recovery of bandlimited graph signals via randomized dynamical sampling, preprint (2021), arXiv:2109.14079.

\bibitem{LV09} Y. Lu and M. Vetterli, Spatial super-resolution of a diffusion field by temporal oversampling in sensor networks, In IEEE International Conference on Acoustics, Speech and Signal Processing (2009), 2249--2252.
		
\bibitem{MMO21} R. D. Mart\'in, I. Medri, and J. Osorio, Error analysis on the initial state reconstruction problem, preprint (2021), arXiv:2105.02015.

\bibitem{NFBK}  B. Sz. Nagy, C. Foias, H. Bercovici, and L. K\'erchy, Harmonic Analysis of Operators on Hilbert Spaces. Springer, 2010.
		
\bibitem{Phi17} F. Philipp, Bessel orbits of normal operators, J. Math. Anal. Appl., {\bf 448}, 2 (2017), 767--785.

\bibitem{RR} H. Radjavi and P. Rosenthal, Invariant Subspaces, Springer-Verlag, Berlin, 1973.

\bibitem{RCLV09} J. Ranieri, A. Chebira, Y.M. Lu, and M. Vetterli, Sampling and reconstructing diffusion fields with localized sources, In IEEE International Conference on
Acoustics, Speech and Signal Processing (ICASSP) (2011), 4016--4019.

\bibitem{Tan17} S. Tang, System identification in dynamical sampling, Adv. Comput. Math., {\bf 43}, 3 (2017), 555--580.



%%%%%%%%%%%%%%%%%%%







%%%%%%%%%%%%%%%%%%%		
	\end{thebibliography}
\end{document}